\documentclass[11pt,a4paper]{amsart}
\usepackage{amsmath,amssymb,amscd,amsthm,amsxtra, mathrsfs, array,soul,ulem}
\usepackage[shortlabels]{enumitem}

\usepackage[latin9]{inputenc}

\usepackage[yyyymmdd,hhmmss]{datetime}
\usepackage[usenames,dvipsnames]{xcolor}
\usepackage[
colorlinks=true,
linkcolor=blue,
citecolor=blue,
urlcolor=blue,
]{hyperref}

\usepackage{ulem}
\usepackage{cancel}
\usepackage{url}

\newtheorem{theorem}{Theorem}[section]
\newtheorem{lemma}[theorem]{Lemma}

\newtheorem{corollary}[theorem]{Corollary}

\theoremstyle{definition}

\theoremstyle{remark}
\newtheorem{remark}[theorem]{Remark}
\numberwithin{equation}{section}

\newcommand{\strt}[1]{\rule{0pt}{#1}}

\newcommand{\Norm}[2]{\|#1\|_{#2}}  

\def\R{{\mathbb R}}

\def \Lpq{L^{p,q}(w)}
\def \d {\mathrm{d}}
\def \ds{\frac{\d s}{s}}
\def \dt{\frac{\d t}{t}}
\def\M {{\mathrm M}}

\newcommand{\RN }{\mathbb R^n}

\newcommand{\cR }{\mathcal R}

\def\Xint#1{\mathchoice
  {\XXint\displaystyle\textstyle{#1}}%
  {\XXint\textstyle\scriptstyle{#1}}%
  {\XXint\scriptstyle\scriptscriptstyle{#1}}%
  {\XXint\scriptscriptstyle\scriptscriptstyle{#1}}%
  \!\int}
\def\XXint#1#2#3{{\setbox0=\hbox{$#1{#2#3}{\int}$}
    \vcenter{\hbox{$#2#3$}}\kern-.5\wd0}}

\def\avgint{\Xint-}
\newcommand{\vertiii}[1]{{\left\vert\kern-0.25ex\left\vert\kern-0.25ex\left\vert #1 
    \right\vert\kern-0.25ex\right\vert\kern-0.25ex\right\vert}}

\numberwithin{equation}{section}

\begin{document}

 \title[  mixed $A_p-A_{\infty}$ estimate for maximal functions]{Weighted Lorentz spaces:  sharp mixed  $A_p-A_{\infty}$ estimate for maximal functions}

\author[N. Accomazzo]{Natalia Accomazzo}
\address[Natalia Accomazzo]{Department of Mathematics, The University of British Columbia, Room 121, 1984 Mathematics Road, Vancouver, BC, Canada V6T 1Z2}
 \email{naccomazzo@math.ubc.ca}

\author[J. Duoandikoetxea]{ Javier Duoandikoetxea}
\address[Javier Duoandikoetxea]{
Sopela, Basque Country, Spain} \email{javier.duoandikoetxea@ehu.eus}

\author[Z.Nieraeth]{Zoe Nieraeth}
\address[Zoe Nieraeth]{BCAM\textendash  Basque Center for Applied Mathematics, Bilbao, Spain}
\email{znieraeth@bcamath.org}

\author[S. Ombrosi]{Sheldy Ombrosi}
\address[Sheldy Ombrosi] {Departmento de Matem\'atica, Universidad Nacional del Sur,
Bah\'ia Blanca, Argentina}
\email{sombrosi@uns.edu.ar}

\author[C. P\'erez]{Carlos P\'erez}
\address[Carlos P\'erez]{ Department of Mathematics, University of the Basque Country, IKERBASQUE 
(Basque Foundation for Science) and
BCAM \textendash  Basque Center for Applied Mathematics, Bilbao, Spain}
\email{cperez@bcamath.org}

\thanks{N. A. is partially supported by NSERC Discovery grant.}
\thanks{Z. N. is supported by the Basque Government through the BERC 2018-2021 program and by the Spanish State Research Agency through BCAM Severo
Ochoa excellence accreditation SEV-2017-0718 and through project PID2020-113156GB-I00/AEI/10.13039/501100011033 funded by Agencia Estatal de
Investigación and acronym ``HAPDE''}
\thanks{S. O. is partially supported by Agencia I+d+i grant PICT 2018-2501}
\thanks{C. P. is supported by grant  PID2020-113156GB-I00, Spanish Government; by the Basque Government through grant IT1247-19 and the BERC 2014-2017 program  and by BCAM Severo Ochoa accreditation SEV-2013-0323, Spanish Government}



\begin{abstract}
We prove the sharp mixed $A_{p}-A_{\infty}$ weighted estimate for the Hardy-Littlewood maximal function in the context of weighted Lorentz spaces, namely
$$
\|\M\|_{\Lpq}
\lesssim_{p,q,n} [w]^{\frac1p}_{A_p}[\sigma]^{\frac1{\min(p,q)}}_{A_{\infty}},
$$
where $\sigma=w^{\frac{1}{1-p}}$. Our method is rearrangement free and can also be used to bound similar operators, even in the two-weight setting. We use this to also obtain new quantitative bounds for the strong maximal operator and for $\M$ in a dual setting.

\end{abstract}
\maketitle

\normalem


\section{Introduction}

We begin by introducing the most classical set up for our problem
 which is provided by the Hardy-Littlewood maximal operator  $\M$ 
\[\M f(x)=\sup_Q \avgint_Q |f(y)| \,\d y,\]
where the supremum is taken over all the cubes $Q$ in $\R^n$ that contain $x$, and the barred integral stands for the average on $Q$.

In the celebrated paper \cite{M}, B. Muckenhoupt characterized the class of weights for which the Hardy--Littlewood
maximal operator is bounded on $L^p(w)$; the surprisingly simple
necessary and sufficient condition is the renowned $A_{p}$
condition of Muckenhoupt, namely $\|\M\|_{L^p(w)}$ is finite if and only if the quantity
\begin{equation}\label{Ap}
[w]_{A_p}:=\sup_{Q}\left(\avgint_Qw(x)\,\d x\right)\left(\avgint_Q w(x)^{1-p'}\,\d x\right)^{p-1}  
\end{equation}
is finite. We will call this quantity the $A_{p}$ constant of the weight, although in the literature it's sometimes been referred as the ``characteristic'' or the ``norm'' of the weight. 

The first result linking the relationship between the operator norm
$\|\M\|_{L^p(w)}$ and the $A_p$ constant of the weight $w$ was obtained by S.~Buckley~\cite{buckley} as part of his Ph.D. thesis. Precisely, if $1<p<\infty$ and $w \in A_p$, then the Hardy-Littlewood maximal function
satisfies the following operator estimate:
\begin{equation}\label{buckley0}
\|\M\|_{L^p(w)}\leq C_{n,p}\,
[w]_{A_p}^{\frac{1}{p-1}}, \qquad  1<p<\infty.
\end{equation}
Moreover, the constant $C_{n,p}$ can be taken as $c_n p'$, where $c_n$ depends only on $n$. Furthermore the result is sharp in the
sense that for any $\epsilon>0$
\begin{equation}\label{buckley3}
\sup_{w\in A_{p}} \frac{1}{[w]_{A_{p}}^{\frac{1}{p-1}-\epsilon}}
\|\M\|_{L^p(w)} =\infty.
\end{equation}
The sharpness of the exponent was also showed by Buckley \cite{buckley} via a specific example. Using a different approach \cite{LPR15} that takes into account the norm behaviour of $\M$ at the endpoint:
\begin{equation*}
\|\M\|_{L^{p}(\R^n) } \approx \frac{1}{p-1},  \qquad p\to 1,
\end{equation*}
it was actually proven that we cannot replace the constant $[w]_{A_p}^{1/(p-1)}$ by any other $\varphi([w]_{A_p})$ with $\varphi$ a non-negative non-decreasing function with $\lim_{t\to\infty}\varphi(t)/t^{p-1}=0$. Unfortunately, the same approach cannot be translated into the case of the strong maximal function $\M_{\cR}$ which we discuss below.

In the same article that we've already mentioned \cite{buckley}, Buckley noted and made use of the fact that for $1\leq p <\infty$,
\begin{equation}\label{Weakp,p}
	    \Vert \M\Vert_{L^{p}(w)\to L^{p,\infty}(w)} \lesssim [w]_{A_p}^{1/p}.
\end{equation}
This result was implicit in the proof of Muckenhoupt in \cite{M}. Moreover it is a characterization of the $A_p$ constant since 
\begin{equation}\label{CharactApWeakp,p}
[w]_{A_p}^{\frac1p}  \leq   \Vert \M\Vert_{L^{p}(w)\to L^{p,\infty}(w)} \leq c_n\, [w]_{A_p}^{\frac1p} 
\end{equation}
This characterization holds true even in the two-weight context,
$$\|\M\|_{L^{p}(w)\to L^{p,\infty}(v)}\approx_n [v,w]_{A_p}^{\frac{1}{p}}.$$
Here, for two weights $v,w$, we say that $(v,w)\in A_p$ when
\[
[v,w]_{\strt{2ex}A_p}:=\sup_{Q}\left(\avgint_Qv(x)\,\d x\right)\left(\avgint_Q w(x)^{1-p'}\,\d x\right)^{p-1}<\infty  .
\]

As was mentioned above, the exponent in Buckley's estimate \eqref{buckley0} cannot be improved. However, \eqref{buckley0} can be improved by inserting another constant related to $w$, more precisely the $A_{\infty}$ constant of the ``dual weight" of $w$, namely $\sigma=w^{-1/(p-1)}$. This was done first in \cite{HP} and a simplified proof was given in  \cite{HPR}.

To define the $A_{{\infty}}$ class of weights a key observation is that the $A_{p}$ condition is decreasing in $p$, by this we mean that 
\begin{equation*}
[w]_{A_q}\leq [w]_{A_p},   \quad 1\leq p<q,
\end{equation*}
implying that  $A_{p}\subset A_{q}$.  Hence it is natural to define 
\begin{equation*}
A_{\infty} := \bigcup_{p>1} A_p.
\end{equation*}
The $A_{\infty}$ class of weights shares a lot of interesting
properties, we remit to \cite{Duo-Paseky} for more information. Although we find the definition of the class very natural, a definition for its corresponding constant is not so much and different attempts have been made through the literature. In \cite{HP} the authors propose to define the $A_\infty$ constant of a weight $w$ as 
 \begin{equation*}
[w]_{A_\infty}:= \sup_Q\frac{1}{w(Q)}\int_Q \M(w\chi_Q )\,\d x.
 \end{equation*}
This corresponds to a characterization of $A_\infty$ originally introduced by Fujii in \cite{fujii} and rediscovered by Wilson in \cite{Wilson:87}. Observe that $[w]_{A_\infty} \ge 1$ by the Lebesgue differentiation theorem. On the other hand, $[w]_{ A_\infty}\le C_n [w]_{ A_p}$ for $w\in A_p$ (see \cite{HP}). For more information about characterizations of $A_\infty$, see \cite{DMRO}.

Now we are in position to state the improvement of Buckley's theorem \eqref{buckley0} obtained in \cite{HP}.

\begin{theorem}\label{MixedApAinftyLp}   Let $1<p<\infty$ and let $\sigma=w^{1-p'}$, then 
\begin{equation}\label{HLLpAinfty}
  \|\M\|_{L^p(w)} \leq c_n\, \big(p'\,  [w]_{ A_p} [\sigma]_{ A_\infty}\big)^{1/p}.
\end{equation}
\end{theorem}
Indeed, using the duality relationship \,$[\sigma]_{ A_{p'}}^{1/p'}=[w]_{ A_p}^{1/p}$ and $[\sigma]_{ A_\infty}\le C_n [\sigma]_{ A_{p'}}$,
we see immediately that 
\begin{equation*}
\big([w]_{ A_p} [\sigma]_{ A_\infty}\big)^{1/p} \leq C_n [w]_{A_p}^{\frac{1}{p-1}},
\end{equation*}
yielding \eqref{buckley0}.  

One really interesting consequence of this result is that it disproves the ``reverse'' $A_2$ conjecture for $\M$, namely that there is a dimensional constant $c_n$ such that for any $w\in A_2$, $ \|\M\|_{L^2(w)} \geq c_n [w]_{ A_2}$. By \eqref{HLLpAinfty}, to disprove the conjecture it is enough to find a family of weights in $A_2$ for which the constant $[\sigma]_{A_{\infty}}$ grows slower than $[w]_{A_2}$. We refer to \cite{HP} for more details.

It is somehow natural that the $A_{\infty}$ constant would appear in estimate \eqref{HLLpAinfty} because  
$w\in A_p$ if and only if both $w$ and $\sigma$ belong in $A_{\infty}$ (see \cite[Chapter IV, Theorem 2.17] {GCRdF}). However we cannot improve \eqref{HLLpAinfty} by replacing $[w]_{ A_p}$ by  $[w]_{ A_\infty}$,  namely, the estimate 
\begin{equation}\label{DoubleAinfty}
\|\M\|_{\strt{2ex}L^p(w)}\le C ([w]_{\strt{2ex}A_\infty}[\sigma]_{\strt{2ex}A_\infty})^{1/p}
\end{equation}
is false (see Section \ref{optimality}). 

A key estimate to derive \eqref{HLLpAinfty} was the following result, also from \cite{HP}. A simplified proof can be found in \cite{HPR}.

\begin{theorem}\label{thm:SharpRHI}
There is a dimensional constant $c_n\geq 1$ such that if $w\in A_{\infty}$ and if $1<r_w<\infty$ is given by 
$$r_w'=c_n[w]_{\strt{2ex}A_\infty},$$ 
then for all cubes $Q$, 
$$
\displaystyle\left( \avgint_Q w^{r_w}\right)^{1/r_w}\leq 2\avgint_Q w.
$$
\end{theorem}

We mention that a similar one-dimensional result was also obtained in \cite{BR} by completely different means. 

The following corollary, which is usually called the {\it openness} property of the $A_{p}$ condition, is an important consequence of the reverse H\"older property. For the purpose of our paper we will state it in its two-weight version.

\begin{corollary}\label{Apselfimprovetwoweight} 
Let $1<p<\infty$ and let $(v,w)\in
A_p$, and $\sigma:=w^{1-p'}\in A_\infty$. Then for $1<r<\infty$ satisfying 
$$r'=c_np'[\sigma]_{\strt{2ex}A_\infty},$$ we have $(v,w)\in A_{\frac{p}{r}},$ with
$$ [v,w]_{\strt{2ex}A_{\frac{p}{r}}} \leq 2^{p-1}\,[v,w]_{\strt{2ex}A_{p}}.$$
\end{corollary}
\begin{proof}
We have $r_\sigma'={r'}/{p'}$ and hence, by Theorem~\ref{thm:SharpRHI}, we have
\[
\left(\avgint_Q w(x)^{1-(\frac{p}{r})'}\,\d x\right)^{\frac{p}{r}-1}=\left(\avgint_Q \sigma(x)^{r_\sigma}\,\d x\right)^{\frac{p-1}{r_\sigma}}\leq 2^{p-1}\left(\avgint_Q \sigma(x)\,\d x\right)^{p-1}.
\]
The result now follows from the definition of the weight constants.
\end{proof}

The proof of \eqref{HLLpAinfty} was simplified in \cite{HPR} using this result
(which can also be found in the Paseky lecture notes \cite{CP7-Paseky}). 
Finally,  it should be mentioned that the reverse H\"older property or the openness property was completely avoided  in \cite{PR} where a more general theorem was proved containing the mixed $A_{p}-A_{\infty}$ result \eqref{HLLpAinfty}, this time using the two weight testing characterization of Sawyer for $\M$ \cite{S} (see also \cite{GCRdF} and \cite{J}).

The main purpose of this paper is to understand most of the results mentioned above but in the context of Lorentz spaces.
 Indeed, motivated by Muckenhoupt's paper \cite{M} 
H. Chung, R. Hunt, and D. Kurtz  improved Muckenhoupt's theorem in \cite{CHK} within the context of Lorentz spaces. 

\begin{theorem} \label{CHK} \
Let $p\in(1,\infty)$ and let $q\in [1,\infty]$, then  
$$
\|\M\|_{\Lpq}<\infty   \qquad \text{if and only if} \qquad  w \in A_p.
$$
\end{theorem}

Our aim with this paper is to extend this theorem in several ways.  
First we will present a general result, Theorem \ref{thm:main}, from which we will be able to deduce, among other things, an improvement of Theorem \ref{CHK} by introducing the $A_p-A_{\infty}$ constants in the spirit of \eqref{HLLpAinfty}. It's worth noting that we also obtain sharpness in the exponents of the constants of the weight. This result is the content of Corollary \ref{MixedApAinftyLpq}.

From Theorem \ref{thm:main} we will also be able to deduce an extension of Theorem \ref{CHK} for the strong maximal operator.
This operator, denoted by $\M_{\cR}$, is defined as
\begin{equation}\label{StrongMF}
\M_{\cR}(f)(x)=\sup_{R\ni x} \avgint_R |f(y)|\,\d y,\,  
\end{equation}
 where  the supremum is taken over all rectangles in $\RN$ with
sides parallel to the coordinate axes.

It is well known that the appropriate class of weights attached to this maximal function is the \emph{strong} class defined analogously to the standard $A_{p}$ class as follows: we will say that $w \in A_{p,\cR}$ if
\begin{equation*}\label{eq:Ap-strong}
[w]_{A_{p,\cR}}:=\sup_{R\in \cR} \left(\avgint_R w(x)\,\d x\right) \left(\avgint_R w(x)^{-\frac{1}{p-1}}\,\d x \right)^{p-1}< \infty.
\end{equation*}
This condition was introduced in \cite{BK} and the following result was essentially known for $p\in(1,\infty)$: 
$$\|\M_{\cR}\|_{L^p(w)}<\infty \quad \Longleftrightarrow \quad  \|\M_{\cR}\|_{L^p(w)\to L^{p,\infty}(w)} <\infty \quad \Longleftrightarrow \quad  w \in A_{p,\cR}.$$

More recently, these qualitative results have been refined, and the following quantitative versions are well known:  
\begin{equation}\label{Lp-pEstimate}
\|\M_{\cR}\|_{L^p(w)} \leq c_n\, (p')^n  [w]^{\frac{n}{p-1}}_{A_{p,\cR}},  \qquad 1<p<\infty,
\end{equation}
and the  corresponding  weak norm estimate, 
\begin{equation}\label{WeakLp-p-Estimate}
\|\M_{\cR}\|_{L^p(w)\to L^{p,\infty}(w)} \leq c_n(p')^{n-1}[w]_{_{A_{p,\cR}}}^{\frac1p+\frac{n-1}{p-1}},     \qquad 1<p<\infty.
\end{equation} 

These results were improved in \cite{LPR17} by  proving mixed $A_{p}-A_{\infty}$ estimates in the spirit of Theorem \ref{MixedApAinftyLp}. To do this, the $A_{\infty,\cR}$ constant was introduced in this work. This class of weights had been already introduced in several papers in the same way as in the cubic case, to wit 
\begin{equation*}
A_{\infty,\cR}=\bigcup_{p>1}A_{p,\cR},
\end{equation*}
and it enjoys similar geometric conditions. Again, there are several different characterizations of this class. To define the associated constant we choose to work with the Fujii-Wilson constant: 
\begin{equation}\label{eq:AinftycR}
   [w]_{A_{\infty,\cR}}:=\sup_{R\in \cR}\frac{1}{w(R)}\int_R \M_{\cR}(w\chi_R )\,\d x <\infty.
\end{equation}

As before a key estimate for the improvement was an appropriate reverse H\"older inequality in this context as shown in \cite{LPR17}.

\begin{theorem}\label{thm:Ainfty-RHI-cR}
Let $w\in A_{\infty,\cR}$ and let $1<\rho_w<\infty$ be given by $$\rho_w'=2^{n+1}[w]_{A_{\infty,\cR}}.$$ Then for all rectangles $R$,
\begin{equation*}
\left(\avgint_{R} w^{\rho_w}\,\d x\right)^{\frac{1}{\rho_w}} \le 2\avgint_{R} w \,\d x.
\end{equation*}
\end{theorem}

This result is interesting for many other reasons, see for instance \cite{CMPR} for the context of Poincaré-Sobolev inequalities in product spaces.

From this result one derives the openness property of these classes of weights in the same way as before.

\begin{corollary} \label{ApselfimproveStrong} 

Let $1<p<\infty$ and let $w\in
{A_{p,\cR}}$. Set $\sigma:=w^{1-p'}$. Then for $1<r<\infty$ satisfying 
$$r'=2^{n+1}p'[\sigma]_{A_{\infty,\cR}}$$
we have $w\in A_{\frac{p}{r},\cR},$ with
$$ [w]_{A_{\frac{p}{r},\cR}} \leq 2^{p-1}\,[w]_{A_{p,\cR}}.$$
\end{corollary}

The following $A_{p}-A_{\infty}$ remark was made in \cite{LPR17}:
\begin{equation}\label{eq:StrongHLLpAinfty} 
  \|\M_{\cR}\|_{L^p(w)} \leq c\, (p')^n  [w]^{\frac{1}{p}+ 2\frac{n-1}{p-1}}_{A_{p,\cR}} [\sigma]_{A_{\infty,\cR}}^{\frac{1}{p}},   \qquad 1<p<\infty.
\end{equation}
as it turns out, the result we obtain improves this.

Finally, we wish to use Theorem~\ref{thm:main} to consider a dual estimate for $M$ on weighted Lorentz spaces. Unlike what happens for Lebesgue spaces, these spaces generally do not dualize into other weighted Lorentz space. Indeed, when viewing the space $X:=L^{p,q}(w)$ for $p,q\in(1,\infty)$ as a Banach function space over the associated unweighted measure space, then its associate space $X'$ is given through
\[
\|f\|_{X'}=\Big\|\frac{f}{w}\Big\|_{L^{p',q'}(w)}.
\]
While in the case $p=q$ we have $X'=L^{p'}(w^{1-p'})$, the situation is more complicated in the general case $p\neq q$. Setting $Y=L^{p,q}(v)$, the dual boundedness $\M:Y'\to X'$ is given by
\begin{align*}
\|\M\|_{Y'\to X'}&=\sup_{\|f\|_{L^{p',q'}(v)}=1}\|\M(fv)\|_{X'}
=\sup_{\|f\|_{L^{p',q'}(v)}=1}\Big\|\frac{\M(fv)}{w}\Big\|_{L^{p',q'}(w)}\\
&=\|T\|_{L^{p',q'}(v)\to L^{p',q'}(w)},
\end{align*}
where $Tf:=\frac{\M(fv)}{w}$. Hence, studying our dual problem is synonymous with studying the boundedness of $T$ on weighted Lorentz spaces. This operator also appears in the so called Sawyer's problem (see \cite{LOP} and the references therein).

\section{Statement of the main results}

We turn our attention to Lorentz spaces. Traditionally, this family of function spaces have been defined in terms of decreasing rearrangements of functions, this is the case in the aforementioned work \cite{CHK}. Here we will be using a definition that makes use of the distribution of a function instead, a definition we find more approachable. For $p,q>0$ and a weight $w$ we say that a function $f$ belongs to the weighted Lorentz space $L^{p,q}(w)$ if 
\[\|f\|_{\Lpq}:= \Bigg(  p\,\int_{0}^{\infty}s^q w\Big(\{x\in \R^n: |f(x)|>s\}\Big)^{q/p} \,\ds \Bigg)^{1/q}<\infty.\]	
Observe that when $q=p$ the Lorentz space $\Lpq$ coincides with the usual $L^p(w)$ space. It's a well known property that the Lorentz spaces are nested, that is $\Lpq \subset L^{p,r}(w)$ if $0<q<r\le\infty$. Moreover, we have the norm inequality $\|f\|_{L^{p,r}(w)}\lesssim_{p,q,r}\|f\|_{\Lpq}$. It is not surprising then that both the case $q=1$ and $q=\infty$, often referred as the weak $L^p(w)$ space, are of special interest in the theory. The space $L^{p,\infty}(w)$ which is not included directly in the definition and is given by
\begin{equation*}
\Vert f \Vert_{L^{p,\infty}(w)}:= \sup_{s>0} s (w\{x\in \R^n: |f(x)|>s\})^{1/p}<\infty. 
\end{equation*}

For a given operator $T$ acting on $\Lpq$, whether it is linear or not, we use the following notation 
\begin{equation*}
   \Norm{T}{L^{p,q}(w)}:=\sup_{f\neq 0}\frac{\Norm{Tf}{L^{p,q}(w)}}{\Norm{f}{L^{p,q}(w)}}.
\end{equation*}

Our main general result is the following.

\begin{theorem}\label{thm:main}
Let $(\Omega,\mu)$ be a $\sigma$-finite measure space. Let $p\in(1,\infty)$, $q\in(0,\infty]$, $r\in(1,p)$, $A>0$,  and let $v,w$ be weights in $\Omega$. Suppose $T$ is a positive sublinear operator in $L^0(\Omega)$ satisfying the following properties:
\begin{enumerate}[(i)]
\item\label{enum:op1} $T$ is defined on the constant function $1$ with $|T1|\leq A1$;
\item\label{enum:op2} $T$ is bounded from $L^{p/r}(w)$ to $L^{p/r,\infty}(v)$.
\end{enumerate}

Then
\begin{align*}
\|T\|_{L^{p,q}(w)\to L^{p,q}(v)}
\lesssim\begin{cases}
(1+A)(r')^{\frac{1}{p}}\|T\|^{\frac{1}{r}}_{L^{p/r}(w)\to L^{p/r,\infty}(v)} & \text{if $p\leq q\leq\infty$;}\\
(1+A)\big(\frac{4r'}{q}\big)^{\frac{1}{q}}\|T\|^{\frac{1}{r}}_{L^{p/r}(w)\to L^{p/r,\infty}(v)} & \text{if $0<q\leq p$.}
\end{cases}
\end{align*}
\end{theorem}

As a corollary we have the promised extension of Theorem \ref{CHK}.

\begin{corollary} \label{MixedApAinftyLpq} \

1) Let $p\in(1,\infty)$, $q\in(0,\infty]$, and $w\in A_p$ with $\sigma:=w^{1-p'}$. Then
$$
\|\M\|_{\Lpq}
\lesssim_{q,n} (p')^{\frac1{\min(p,q)}}\,[w]^{\frac1p}_{A_p}[\sigma]^{\frac1{\min(p,q)}}_{A_{\infty}}.
$$

2) Furthermore the result is optimal, namely if we replace any of the exponents by 
a smaller one, then the result is false.

\end{corollary}

We will actually prove the following more precise two-weight version of the first part of Corollary \ref{MixedApAinftyLpq}:

\begin{corollary}\label{thm:maintwoweight} 
Let $p\in(1,\infty)$, $q\in(0,\infty]$ and let $(v,w)\in A_p$ with $\sigma:=w^{1-p'}\in A_\infty$.  Then
\begin{align*}
\|\M\|_{L^{p,q}(w)\to L^{p,q}(v)}
\lesssim\begin{cases}
c_np'[v,w]^{\frac{1}{p}}_{A_p}[\sigma]_{A_\infty}^{\frac{1}{p}} & \text{if $p\leq q\leq\infty$;}\\
\big(\frac{4c_n}{q}\big)^{\frac{1}{q}}(p')^{\frac{1}{q}}[v,w]^{\frac{1}{p}}_{A_p}[\sigma]_{A_\infty}^{\frac{1}{q}} & \text{if $0<q\leq p$.}
\end{cases}
\end{align*}
\end{corollary}

Our method of proof is completely different from the one in \cite{CHK}. Their method is based on  the concept of rearrangement, but it is not clear how to use it to derive our much more precise estimates. Our method is more transparent and it is in the spirit of the proof of \cite{HPR} using the reverse H\"older property of the $A_{\infty}$ weight $\sigma$.

As promised, we also analyze the case of the strong maximal operator. From Theorem \ref{thm:main} we can deduce the following corollary:

\begin{corollary} \label{MixedApAinftyLpqStrg}  Let $p\in(1,\infty)$, $q\in (0,\infty]$,  and let $w\in A_{p,\cR}$ with $\sigma:=w^{1-p'}$. Then
\[
\|\M_\cR\|_{L^{p,q}(w)}
\lesssim_{q,n}
(p')^{n-1+\frac1{\min(p,q)}} \log(e+[w]_{A_{p,\cR}})^{\frac1{\min(p,q)}}[w]^{\frac{1}{p}+\frac{n-1}{p-1}}_{A_{p,\cR}}[\sigma]_{A_{\infty,\cR}}^{\frac1{\min(p,q)}}.
\]

\end{corollary}

Observe that this result is better than the case $p=q$ in \eqref{eq:StrongHLLpAinfty}. As a matter of fact, since
\[
[w]^{\frac{1}{p}+\frac{n-1}{p-1}}_{A_{p,\cR}}[\sigma]_{A_{\infty,\cR}}^{\frac1p}\lesssim[w]^{\frac{n}{p-1}}_{A_{p,\cR}},
\]
we only lose a $\log$-term in our estimate compared to \eqref{Lp-pEstimate}.We believe that in the non mixed situation, the results  \eqref{Lp-pEstimate} and \eqref{WeakLp-p-Estimate} are sharp. This loss of exponent compared to the cubic case could be related to the lack of sparse domination in this situation as shown in \cite{BCOR}.

Finally, we consider the dual problem to Corollary~\ref{thm:maintwoweight}. We obtain the following result:
\begin{corollary}\label{thm:maindual}
Let $p,q\in(1,\infty)$ and let $(v,w)\in A_1$ with $w\in A_1$.  Then the operator $Tf:=\frac{\M(fv)}{w}$ satisfies
\begin{align*}
\|T\|_{L^{p',q'}(v)\to L^{p',q'}(w)}
\lesssim_n
p^{\frac{1}{\min(p',q')}}\log(e+[v,w]_{A_1})^{\frac{1}{\min(p',q')}}(1+ [v,w]_{A_1})[w]_{A_1}^{\frac{1}{p'}}.
\end{align*}
\end{corollary}
\begin{remark}
We note that in the case $v=w$, $p=q$, $T$ is bounded on $L^{p'}(w)$ if and only if $\M$ is bounded on $L^{p'}(w^{1-p'})$. This is equivalent to the assumption $w^{1-p'}\in A_{p'}$ (or equivalently, $w\in A_p)$,  and hence, does not require $w\in A_1$. 
It is not clear if it is true in the general case $p\neq q$ if $T$ is bounded under the weaker assumption $w\in A_p$. We note here that property (i) in Theorem~\ref{thm:main} for this operator in this case is equivalent to the condition $w\in A_1$ and hence, other methods would be required to study this problem.
\end{remark}

We finish this section by mentioning that as usual we write $A \lesssim B$ if there exists a positive constant $C>0$, independent of $A$ and $B$, such that $A\leq C B$. If the implicit constant $C$ depends on some parameter $\alpha$, we may write $\lesssim_\alpha$ at our discretion. If $A\lesssim B$ and $B\lesssim A$, then we write $A\approx B$.

\section{Proof of the main results}	

For the proof of Theorem~\ref{thm:main} we first prove a lemma that will be used for the cases $0<q<p$.
\begin{lemma}\label{lem:qlep}
Let $\theta\in(0,1]$,  $t\in(0,\infty)$, and let $\phi$ be increasing and $\psi$ be decreasing. Then
\[
\left(\int_t^\infty\!\phi(s)\psi(s)\,\frac{\mathrm{d}s}{s}\right)^\theta\leq(\log 2)^{-(1-\theta)}\int_{\frac{t}{2}}^\infty\!\phi(4s)^\theta\psi(s)^\theta\,\frac{\mathrm{d}s}{s}.
\]
\end{lemma}

\begin{proof}

Since $\phi$ is increasing and $\psi$ is decreasing, it follows that for $s\in (a,2a)$ we have $\phi(a)\leq\phi(s)\leq\phi(2a)\leq\phi(2s)$ and $\psi(2a)\leq\psi(s)\leq\psi(a)\leq\psi(s/2)$. Hence, using the fact that $\|\cdot\|_{\ell^{\theta}}\leq\|\cdot\|_{\ell^1}$, we have
\begin{align*}
&\left(\int_t^\infty\!\phi(s) \psi(s)\,\ds\right)^\theta=\left(\sum_{k=0}^\infty\int_{2^kt}^{2^{k+1}t}\!\phi(s)\psi(s)\,\ds\right)^\theta\\
&\qquad \leq\left(\sum_{k=0}^\infty\phi(2^{k+1}t)\psi(2^kt)\log 2\right)^\theta
\leq(\log 2)^\theta\sum_{k=0}^\infty\phi(2^{k+1}t)^\theta\psi(2^kt)^\theta\\
&\qquad=(\log 2)^{\theta-1}\sum_{k=0}^\infty\int_{2^kt}^{2^{k+1}t}\!\phi(2^{k+1}t)^\theta\psi(2^kt)^\theta\,\ds\\
&\qquad\leq(\log 2)^{\theta-1}\int_t^\infty\!\phi(2s)^\theta\psi(2^{-1}s)^\theta\,\ds
=(\log 2)^{\theta-1}\int_{\frac{t}{2}}^\infty\!\phi(4s)^\theta\psi(s)^\theta\,\ds.
\end{align*}

\end{proof}

\begin{proof}[Proof of Theorem~\ref{thm:main}]
For $t>0$ we set $f_t:=f\chi_{\{|f|>t\}}$. Then, since $|f|=|f_t|+|f|\chi_{\{|f|\leq t\}}\leq|f_t|+t$, it follows from property \ref{enum:op1} that $|T f|\leq T(|f_t|)+At$ and thus $\{|T f|>(1+A)t\}\subseteq\{T(|f_t|)>t\}$. We also note that for $s>0$ we have $|f_t|>s$ if and only if $|f|>\max\{t,s\}$. 

Next, by property \ref{enum:op2} we may set
\begin{equation*}
N:=\|T\|^{\frac{1}{r}}_{L^{p/r}(w)\to L^{p/r,\infty}(v)}<\infty.
\end{equation*}
Then, fixing $t>0$, we have
\begin{align*}
& t\big(v\{|Tf|>(1+A)t\}\big)^{\frac{1}{p}}
\leq t\big(v\{T(|f_t|)>t\}\big)^{\frac{1}{p}}
\leq t\big(\frac{N^p}{t^{\frac{p}{r}}}\|f_t\|_{L^{\frac{p}{r}}(w)}^{\frac{p}{r}}\big)^{\frac{1}{p}}\\
&\qquad \leq Nt^{\frac{1}{r'}}\left(\frac{p}{r}\int_0^\infty\!s^{\frac{p}{r}}w\{|f_t|>s\}\,\ds\right)^{\frac{1}{p}}\\
&\qquad = Nt^{\frac{1}{r'}}\Big(\frac{p}{r}\Big)^{\frac{1}{p}}\left(\int_0^\infty\!s^{\frac{p}{r}}w\{|f|>\max\{s,t\}\}\,\ds\right)^{\frac{1}{p}}\\
&\qquad =Nt^{\frac{1}{r'}}\Big(\frac{p}{r}\Big)^{\frac{1}{p}}\left(w\{|f|>t\}\int_0^t\!s^{\frac{p}{r}}\,\frac{\mathrm{d}s}{s}+\int_t^\infty\!s^{\frac{p}{r}}w\{|f|>s\}\,\ds\right)^{\frac{1}{p}}\\
&\qquad =N\Big(\frac{p}{r}\Big)^{\frac{1}{p}}\left(\frac{r}{p}t^pw\{|f|>t\}+t^{\frac{p}{r'}}\int_t^\infty\!s^{\frac{p}{r}}w\{|f|>s\}\,\ds\right)^{\frac{1}{p}}.
\end{align*}
Hence, by the (quasi-)triangle inequality, 
\begin{equation}\label{eq:insidetwowlorentz}
\begin{split}
\|&Tf\|_{L^{p,q}(v)}=p^{\frac{1}{q}}\|(1+A)t\big(v\{|Tf|>(1+A)t\}\big)^{\frac{1}{p}}\|_{L^q((0,\infty),\frac{\mathrm{d}t}{t})}\\
&\leq (1+A)N\Big(\frac{p}{r}\Big)^{\frac{1}{p}}p^{\frac{1}{q}}\Big\|\frac{r}{p}t^pw\{|f|>t\}+t^{\frac{p}{r'}}\int_t^\infty\!s^{\frac{p}{r}}w\{|f|>s\}\,\frac{\mathrm{d}s}{s}\Big\|^{\frac{1}{p}}_{L^{\frac{q}{p}}((0,\infty),\frac{\mathrm{d}t}{t})}\\
&\leq2^{(\frac{1}{q}-\frac{1}{p})_+}(1+A)N\left(\|f\|^p_{L^{p,q}(w)}
+\frac{p}{r}p^{\frac{p}{q}}\Big\|t^{\frac{p}{r'}}\int_t^\infty\!s^{\frac{p}{r}}w\{|f|>s\}\,\frac{\mathrm{d}s}{s}\Big\|_{L^{\frac{q}{p}}((0,\infty),\frac{\mathrm{d}t}{t})}\right)^{\frac{1}{p}}\\
&=:2^{(\frac{1}{q}-\frac{1}{p})_+}(1+A)N\left(\|f\|^p_{L^{p,q}(w)}
+I\right)^{\frac{1}{p}}. 
\end{split}
\end{equation}
 (The notation $a_+$ stands for $\max (a,0)$.) We now separate the cases $p\leq q\leq\infty$ and $0<q\leq p$. 

\textit{Case $p\leq q\leq\infty$.} 
Using Minkowski's integral inequality we have
\begin{equation*}
\aligned
I & = \frac{p}{r}p^{\frac{p}{q}}\left(\int_0^\infty t^{\frac{q}{r'}}\left(\int_t^\infty s^{\frac{p}{r}}w\{|f|>s\}\,\ds\right)^{\frac qp} \dt \right)^{\frac pq} \\
&=  \frac{p}{r}p^{\frac{p}{q}}\left(\int_0^\infty t^{\frac{q}{r'}}\left(\int_1^\infty (st)^{\frac{p}{r}}w\{|f|>st\}\,\ds\right)^{\frac qp} \dt \right)^{\frac pq} \\
& \leq\frac{p}{r}p^{\frac{p}{q}} \int_1^\infty s^{\frac{p}{r}} \left(\int_0^\infty t^{\frac{q}{r'}+\frac pr\frac qp} w\{|f|>st\}^{\frac qp}\dt\right)^{\frac pq} \ds \\
&= \frac{p}{r}p^{\frac{p}{q}} \int_1^\infty s^{\frac{p}{r}} \left(\int_0^\infty  \left(\frac ts\right)^q  w\{|f|>t\}^{\frac qp}\dt\right)^{\frac pq} \ds\\
&= \frac{p}{r}p^{\frac{p}{q}} \left(\int_1^\infty s^{-\frac{p}{r'}} \ds\right)\, \frac 1{p^{\frac pq}} \|f\|_{L^{p,q}(w)}^p =(r'-1)\|f\|_{L^{p,q}(w)}^p.
\endaligned
\end{equation*}

Combining this with \eqref{eq:insidetwowlorentz} yields
\[
\|Tf\|_{L^{p,q}(v)}\leq (1+A)N(r')^{\frac{1}{p}}\|f\|_{L^{p,q}(w)},
\]
as asserted.

\textit{Case $0<q\leq p$.}  Applying Lemma~\ref{lem:qlep} with $\theta=\frac{q}{p}$, $\phi(s)=s^p$, $\psi(s)=s^{-\frac{p}{r'}}w\{|f|>s\}$ yields
\begin{align*}
I&=\frac{p}{r}p^{\frac{p}{q}}\left(\int_0^\infty\!t^{\frac{q}{r'}}\left(\int_t^\infty\!s^ps^{-\frac{p}{r'}}w\{|f|>s\}\,\frac{\mathrm{d}s}{s}\right)^{\frac{q}{p}}\,\frac{\mathrm{d}t}{t}\right)^{\frac{p}{q}}\\
&\leq(\log 2)^{-(\frac{p}{q}-1)}4^p\frac{p}{r}p^{\frac{p}{q}}\left(\int_0^\infty\!t^{\frac{q}{r'}}\int_{\frac{t}{2}}^\infty\!\big(s^ps^{-\frac{p}{r'}}w\{|f|>s\}\big)^{\frac{q}{p}}\,\frac{\mathrm{d}s}{s}\,\frac{\mathrm{d}t}{t}\right)^{\frac{p}{q}}\\
&=(\log 2)^{-(\frac{p}{q}-1)}4^p\frac{p}{r}p^{\frac{p}{q}}\left(\int_0^\infty\!s^{-\frac{q}{r'}}\big(s^pw\{|f|>s\}\big)^{\frac{q}{p}}\int_0^{2s}\!t^{\frac{q}{r'}}\,\frac{\mathrm{d}t}{t}\,\frac{\mathrm{d}s}{s}\right)^{\frac{p}{q}}\\
&=(\log 2)^{-(\frac{p}{q}-1)}4^p2^{\frac{p}{r'}}\frac{p}{r}\left(\frac{r'}{q}\right)^{\frac{p}{q}}\|f\|^p_{L^{p,q}(w)}. 
\end{align*}
Hence, by combining this with \eqref{eq:insidetwowlorentz} and noting that $(\log 2)^{-1}\leq 2$ and $p^{1/p}\leq e^{1/e}$, we have
\[
\|Tf\|_{L^{p,q}(v)}\lesssim (1+A)N\left(\frac{4r'}{q}\right)^{\frac{1}{q}} \|f\|_{L^{p,q}(w)}. 
\]
The assertion follows. 
\end{proof}

\begin{proof}[Proof of Corollary \ref{thm:maintwoweight}]
By Corollary~\ref{Apselfimprovetwoweight}  we have 
\[
\|\M\|^{\frac{1}{r}}_{L^{p/r}(w)\to L^{p/r,\infty}(v)}\approx_n [v,w]^{\frac{1}{p}}_{A_{\frac{p}{r}}}\leq2^{\frac{1}{p'}}[v,w]^{\frac{1}{p}}_{A_p}
\]
for $r'=c_np'[\sigma]_{A_\infty}$.  Using that then 
\begin{equation*}
(r')^{\frac{1}{p}}=c_n^{\frac{1}{p}}(p')^{\frac{1}{p}}[\sigma]_{A_\infty}^{\frac{1}{p}}\leq c_n p'[\sigma]_{A_\infty}^{\frac{1}{p}}
\end{equation*}
 and 
 \begin{equation*}
(r')^{\frac{1}{q}}=c_n^{\frac{1}{q}}(p')^{\frac{1}{q}}[\sigma]_{A_\infty}^{\frac{1}{q}},
\end{equation*}
the assertion follows from Theorem~\ref{thm:main} applied with $T=M$, $A=1$. 
\end{proof}

Similarly,  we can prove the following more precise version of Corollary~\ref{MixedApAinftyLpqStrg} for the strong maximal operator:

\begin{corollary}\label{cor:maintstrong}
Let $p\in(1,\infty)$, $q\in(0,\infty]$ and let $w\in A_{p,\cR}$ with $\sigma:=w^{1-p'}$.  Then
\begin{align*}
&\|\M_\cR\|_{L^{p,q}(w)}\\
&\lesssim_n\begin{cases}
(p')^n \log(e+[w]_{A_{p,\cR}})^{\frac{1}{p}}[w]^{\frac{1}{p}+\frac{n-1}{p-1}}_{A_{p,\cR}}[\sigma]_{A_{\infty,\cR}}^{\frac{1}{p}} & \text{if $p\leq q\leq\infty$;}\\
\big(\frac{2^{n+3}}{q}\big)^{\frac{1}{q}}(p')^{\frac{1}{q}+(n-1)}
\log(e+[w]_{A_{p,\cR}})^{\frac{1}{q}}[w]^{\frac{1}{p}+\frac{n-1}{p-1}}_{A_{p,\cR}}[\sigma]_{A_{\infty,\cR}}^{\frac{1}{q}} & \text{if $0<q\leq p$.}
\end{cases}
\end{align*}
\end{corollary}

\begin{proof}
Fix $t\geq 2^{n+1}$. Then, by \eqref{WeakLp-p-Estimate} and Corollary~\ref{ApselfimproveStrong}  we have
\begin{equation}\label{eq:maxweakeststrong}
\|\M_{\cR}\|^{\frac{1}{r}}_{L^{p/r}(w)\to L^{p/r,\infty}(w)} \lesssim_n \Big[\Big(\frac{p}{r}\Big)'\Big]^{\frac{n-1}{r}}[w]_{A_{p,\cR}}^{\frac{1}{p}+\frac{n-1}{p-r}}
\end{equation}
whenever $1<r<\infty$ satisfies $r'=tp'[\sigma]_{A_{\infty,\cR}}$. Note that then $r'/p'\geq t$ and hence
\begin{equation}\label{eq:maxweakeststrong0}
\Big(\frac{p}{r}\Big)'=\Big(\frac{r'}{p'}\Big)'(p'-1)+1\leq t'(p'-1)+1. 
\end{equation}

Since $t\geq 2$, this implies that
\begin{equation}\label{eq:maxweakeststrong1}
\Big[\Big(\frac{p}{r}\Big)'\Big]^{\frac{n-1}{r}}\leq 2^{n-1}(p')^{n-1}.
\end{equation}
Moreover, \eqref{eq:maxweakeststrong0} implies that
\begin{align*}
\frac{1}{p}+\frac{n-1}{p-r}&=\frac{1}{p}\Big(1+\Big(\frac{p}{r}\Big)'(n-1)\Big)\leq\frac{1}{p}\Big(1+\big(t'(p'-1)+1\big)(n-1)\Big)\\
&=\frac{1}{p}+\frac{n-1}{p-1}+\frac{(t'-1)(n-1)}{p(p-1)}.
\end{align*}

By combining this with \eqref{eq:maxweakeststrong} and \eqref{eq:maxweakeststrong1}, it follows from Theorem~\ref{thm:main} applied with $T=\M_{\cR}$ and $A=1$ that
\begin{align*}
\|&\M_\cR\|_{L^{p,q}(w)}
\\
&\lesssim_n \begin{cases}
(p')^n t^{\frac{1}{p}}[w]^{\frac{1}{p}+\frac{n-1}{p-1}+ \frac{(t'-1)(n-1)}{p(p-1)}}_{A_{p,\cR}}[\sigma]_{A_{\infty,\cR}}^{\frac{1}{p}} & \text{if $p\leq q\leq\infty$;}\\
4^{\frac{1}{q}-\frac{1}{p}}\big(\frac{t}{q}\big)^{\frac{1}{q}}(p')^{\frac{1}{q}+(n-1)}[w]^{\frac{1}{p}+
\frac{t'-1}{p(p-1)}(n-1)}_{A_{p,\cR}}[\sigma]_{A_{\infty,\cR}}^{\frac{1}{q}} & \text{if $0<q\leq p$.}
\end{cases}
\end{align*}
Setting $t=2^{n+1}\log(e+[w]_{A_{p,\cR}})+1$ so that 
\begin{equation*}
t'-1=\frac{1}{2^{n+1}\log(e+[w]_{A_{p,\cR}})},
\end{equation*}
the result now follows from the observation that $
[w]_{A_{p,\cR}}^{\frac{1}{\log(e+[w]_{A_{p,\cR}})}}\leq e$. 
\end{proof}

For the proof of Corollary~\ref{thm:maindual} we need an auxiliary result.
\begin{lemma}\label{lem:maindual}
Let $w$ be a weight and let $p\in(0,\infty)$. Then the operator
\[
M^{w,c}_pf(x):=\sup_{r>0}\left(\frac{1}{w(Q_r(x))}\int_{Q_r(x)} |f(y)|^p \, w(y) dy,\right)^{\frac{1}{p}}
\]
where $Q_r(x)$ is the cube centered at $x$ of side-length is $r$, is bounded $L^p(w)\to L^{p,\infty}(w)$ with a dimensional bound, independent of $w$.
\end{lemma}
\begin{proof}
Since $\|M^{w,c}_pf\|_{L^{p,\infty}(w)}=\|M^w_c(|f|^p)\|_{L^{1,\infty}(w)}^{\frac{1}{p}}$, where $M^c_w$ is the centered maximal operator with respect to the measure $w$, this result follows from the usual argument using the Besicovitch covering theorem.
\end{proof}
\begin{proof}[Proof of Corollary~\ref{thm:maindual}]
Since $M (f)\approx_n \M^c(f)=\M^c(|f|)$, it suffices to prove the result for $\M^c$ instead of $\M$, and with $f\geq 0$.  Writing $\avgint_Q u:={\frac{1}{|Q|}}\int_Qu(x) dx$, we have,
\begin{align*}
\avgint_Q fv &\leq    
\left( \avgint_Q f^{p'}v \right )^{\frac{1}{p'}} 
\left( \avgint_Q v \right )^{\frac{1}{p}} 
\\
&=\left( \frac{1}{w(Q)}\int_Q   \Big(f\frac{v^{\frac{1}{p'}}}{w^{\frac{1}{p'}}}\Big)^{p'}  \,wdx\right )^{\frac{1}{p'}} 
\left( \avgint_Q w \right )^{\frac{1}{p'}} 
\left( \avgint_Q v \right )^{\frac{1}{p}}.
\end{align*}
This implies
\[
\M^c(fv) \leq \M_{p'}^{c,w}\Big( f\frac{v^{\frac{1}{p'}}}{w^{\frac{1}{p'}}}\Big)(\M^cw)^{\frac{1}{p'}}(\M^cv)^{\frac{1}{p}}\leq[w]^{\frac{1}{p'}}_{A_1}[v,w]_{A_1}^{\frac{1}{p}}\M_{p'}^{c,w}\Big( f\frac{v^{\frac{1}{p'}}}{w^{\frac{1}{p'}}}\Big)w,
\]
so that by Lemma~\ref{lem:maindual} we have
\begin{align*}
\|Tf\|_{L^{p',\infty}(w)}
&\leq c_n[w]^{\frac{1}{p'}}_{A_1}[v,w]_{A_1}^{\frac{1}{p}}\,
\Big\|\M_{p'}^{c,w}\Big( f\frac{v^{\frac{1}{p'}}}{w^{\frac{1}{p'}}}\Big)\Big\|_{L^{p',\infty}(w)}\\
&\lesssim_n[w]^{\frac{1}{p'}}_{A_1}[v,w]_{A_1}^{\frac{1}{p}}\|f\|_{L^{p'}(v)}.
\end{align*}
Replacing $p'$ by $p'/r$, we obtain
\[
\|T\|^{\frac{1}{r}}_{L^{p'/r}(v)\to L^{p'/r,\infty}(w)}\lesssim_n[w]^{\frac{1}{p'}}_{A_1}[v,w]_{A_1}^{\frac{1}{r}-\frac{1}{p'}}.
\]
Setting
\[
\frac{1}{r}=\frac{1}{p'}+\frac{1}{2p\log(e+[v,w]_{A_1})},
\]
we have $[v,w]_{A_1}^{\frac{1}{r}-\frac{1}{p'}}\leq e$ so that $\|T\|^{\frac{1}{r}}_{L^{p'/r}(v)\to L^{p'/r,\infty}(w)}\lesssim_n[w]^{\frac{1}{p'}}_{A_1}$, and
\[
r'=p\frac{2\log(e+[v,w]_{A_1})}{2\log(e+[v,w]_{A_1})-1}\leq 2p\log(e+[v,w]_{A_1}).
\]
Since $T1\leq[v,w]_{A_1}1$, the result now follows from applying Theorem~\ref{thm:main} with $A=[v,w]_{A_1}$.
\end{proof}

\section{Optimality}\label{optimality}

\begin{proof} [Proof of Corollary \ref{MixedApAinftyLpq} Part 2): Sharpness]
\par	
We are going to prove that our estimate is sharp in the exponent of the weight constants. For this, we will distinguish two cases $p\le q$ and $q\le p$. To simplify notation we are going to give our examples in one dimension, the corresponding $n$-dimensional examples being analogous.

\textit{The case $p\le q$.} We find that Buckley's example \cite{buckley} also works in our case. Define the weights $w_{\delta}(x)=|x|^{(p-1)(1-\delta)}$ for $0<\delta<1$. Then the weights $w_{\delta}$ belong to the $A_p$ class with $[w_{\delta}]_{A_p}\approx\delta^{1-p}$ for every $1<p<\infty$. We also have that $\sigma_\delta(x)=|x|^{\delta-1}$ satisfies $[\sigma_\delta]_{A_\infty}\approx \delta^{-1}$.

Take $f_{\delta}(x)=|x|^{-1+\delta}\chi_{[0,1]}(x)$. This is in $L^{p,q}(w_{\delta})$ for $p\le q$,  because it is in $L^p(w_{\delta})$. It is not difficult to check that the functions $f_{\delta}$ satisfy that $\M f(x)\geq (1/\delta)f(x)$ and so by combining this with Corollary \ref{MixedApAinftyLpq} we get
\[\delta^{-p}\leq \frac{\|\M f\|^p_{L^{p,q}(w_{\delta})}}{\|f\|^p_{L^{p,q}(w_{\delta})}}\lesssim [w_{\delta}]_{A_p}[\sigma_\delta]_{A_{\infty}}\approx \delta^{-p},
\]
which proves the claim. 

\textit{The case $q\le p$.}  For $0<\delta<1$, take $I=[-1,1]$ and 
\begin{equation*}
w_{\delta}(x):= \delta^{p-1}\chi_{I}(x)+|x|^{(1-\delta)(p-1)}\chi_{I^c}(x).
\end{equation*}
Then we will show that $[w_{\delta}]_{A_p}\approx \delta^{-(p-1)}$. We first show the $\geq$ part. Indeed,
taking the average in the interval $[0,2]$ we have
\begin{equation*}
[w_{\delta}]_{A_p}
\ge \frac{1}{2^p} \Big(\int_1^2 x^{(1-\delta)(p-1)}  \d x\Big)\Big(\int_0^1 \frac{1}{\delta}\, \d x\Big)^{p-1} \gtrsim \delta^{-(p-1)},
\end{equation*}
using that $x^{(1-\delta)(p-1)}\ge 1$ for $x\in (1,2)$.

On the other hand, let us take $Q$ to be an interval.  Observe that since the function $w_{\delta}(x)$ is even, we can consider the intervals $Q\subset (0,\infty)$.  Moreover, since $w_{\delta}(x)$ is constant in $I$ and the $A_p$-constant of the weight $|x|^{(p-1)(1-\delta)}$ is similar to $\delta^{1-p}$ it is enough to check  the conditions for intervals $Q\subset (0,\infty)$ such that  $1 \in {Q}$. In these conditions we will consider two cases: 

\begin{itemize}
	\item Let $|Q|\le 1$. In this case, $Q\subset (0,2)$. Then we have that $w_{\delta}(x)\lesssim 1$ and $w^{1-p'}_{\delta}(x)\lesssim 1/\delta$, so  
	  
	\[
	\Big(\avgint_{Q} w_{\delta}(x) \, \d x\Big)\Big(\avgint_{Q} w^{1-p'}_{\delta}(x) \, \d x\Big)^{p-1}\lesssim \delta^{-(p-1)}.\]
	\item Let $|Q|>1$. For this case, we are going to use that $Q\subset (0, 2|Q|)$. Then, we have  
	
	$$\avgint_{Q} w_{\delta}(x) \, \d x \le \frac{1}{|Q|}\Big(\delta^{p-1} + \int_1^{2|Q|} x^{(1-\delta)(p-1)}\, \d x \Big) \lesssim |Q|^{(1-\delta)(p-1)},$$

and  	$$\Big(\avgint_{Q} w^{1-p'}_{\delta}(x) \, \d x\Big)^{p-1}\le \frac{1}{|Q|^{p-1}} \Big(\frac{1}{\delta} +\int_{1}^{2|Q|} x^{\delta-1}\, \d x\Big)^{p-1}\lesssim \frac{1}{\delta^{(p-1)}} |Q|^{(\delta-1)(p-1)}. $$

\end{itemize} 
Combining previous estimates, we conclude that $[w_{\delta}]_{A_p}\approx \delta^{-(p-1)}$ for $0<\delta<1$. 

Now take $f=\chi_{[0,1]}$. Then, we have
\begin{equation*}
\aligned
\|f\|_{L^{p,q}(w_{\delta})}&=p^{1/q}\Big(\int_0^{\infty}t^q w_{\delta}(\{x\in \R: f(x)>t\})^{q/p}\, \dt\Big)^{1/q}\\
&= p^{1/q} \Big(\int_0^1 t^q\, w_\delta(0,1)^{q/p}\dt\Big)^{1/q}=\left(\frac pq\right)^{1/q}\delta^{(p-1)/p}\approx \delta^{1/p'}.
\endaligned
\end{equation*}
We find a lower bound for the norm of $\M f$ using that $\M f(x)=\frac1x$ for $x>1$.
\[\begin{split}
\|\M f\|_{L^{p,q}(w_{\delta})}&= p^{1/q}\Big(\int_0^{\infty}t^qw_{\delta}(\{x\in\R: \M f(x)>t\})^{q/p}\, \dt\Big)^{1/q}\\&\ge p^{1/q}\Big(\int_0^1 t^{q}w_{\delta}(\{x>1: \frac{1}{x}>t\})\, \dt\Big)^{1/q}\\
&=p^{1/q}\Big(\int_0^1 t^qw_{\delta}([1,t^{-1}])^{q/p}\, \dt\Big)^{1/q}\\
&=p^{1/q}\Big(\int_0^1t^q\Big[\int_1^{t^{-1}}x^{(p-1)(1-\delta)}\, \d x\Big]^{q/p}\, \dt\Big)^{1/q}\\
	&\ge p^{1/q}\Big(\int_0^1 t^q\Big[\frac{1}{p}t^{-(p-p\delta+\delta)}\Big]^{q/p}\, \dt\Big)^{1/q}\\
&=p^{1/q}p^{-1/p}\Big(\int_0^1 t^{q\delta(1-1/p)}\, \dt\Big)^{1/q}\approx \Big(\frac1{\delta}\Big)^{1/q}.
\end{split}\]

Putting together all the estimates that we obtained, taking into account that 
$[\sigma_{\delta}]_{A_{\infty}}\le [w_{\delta}]^{1/(p-1)}_{A_p} \sim {\delta}^{-1}$, and applying the first part of Corollary \ref{MixedApAinftyLpq}, we have 
\[\Big(\frac1{\delta}\Big)^{1/q} \,\lesssim \|\M f\|_{L^{p,q}(w_{\delta})}\le [w_{\delta}]^{1/p}_{A_p}[\sigma_{\delta}]_{A_{\infty}}^{1/ q}\|f\|_{L^{p,q}(w_{\delta})}\lesssim \delta^{-1/p'}\delta^{-\frac1q}\delta^{1/p'}=\delta^{-1/q},\]
which proves the claim in the case $q\le p$. 
\end{proof}

We finish the paper by showing that \eqref{DoubleAinfty} is false. In one dimension, take $w(x)=|x|^{(p-1)(1-\delta)}$, with small $\delta>0$. Then, 
\begin{equation*}
[w]_{A_p}\sim \delta^{1-p}, \quad [w]_{A_\infty}\sim 1, \quad [\sigma]_{A_\infty}\sim \delta^{-1}.
\end{equation*}

Since $\|\M\|_{L^p(w)}\sim \delta^{-1}$, in \eqref{HLLpAinfty} one cannot replace   $[w]_{A_p}$ by $[w]_{A_\infty}$. In fact even $[w]_{A_q}$ with $q>p$ cannot be used.

\bibliographystyle{amsalpha}

\end{document}